\documentclass[12pt, a4paper]{amsart} 

\title{The framed little 2-discs operad and diffeomorphisms of
  handlebodies}

\author{Jeffrey Giansiracusa} 
\address{Department of Mathematics, Swansea University\\
 Singleton Park \\ Swansea, Wales \\ SA2 8PP \\ United Kingdom} 
\email{j.h.giansiracusa@swan.ac.uk}
\date{\today}

  \usepackage[pdftex]{graphicx}

\usepackage{mathptmx}
\DeclareSymbolFont{cmlargesymbols}{OMX}{cmex}{m}{n}
\DeclareMathSymbol{\mycoprod}{\mathop}{cmlargesymbols}{"60}
\let\coprod\mycoprod

\usepackage{mathrsfs}
\usepackage{amssymb}
\usepackage{latexsym}
\usepackage{amsmath}
\usepackage{eucal}

\usepackage[protrusion=true,expansion=true]{microtype}

\usepackage[cmtip,arrow]{xy}
\usepackage{pb-diagram,pb-xy}
\usepackage[vmargin=3.2cm, hmargin=3.4cm]{geometry}
\numberwithin{equation}{subsection}

\newtheorem{thmA}{Theorem}

\newtheorem{theorem}{Theorem}[subsection]  
\newtheorem{lemma}[theorem]{Lemma} 
\newtheorem{proposition}[theorem]{Proposition}
\newtheorem{corollary}[theorem]{Corollary}
\newtheorem{definition}[theorem]{Definition}

\theoremstyle{remark} 
\newtheorem{remark}[theorem]{Remark}

\newcommand{\defeq}{\mathbin{\mathpalette{\vcenter{\hbox{$:$}}}=}}

\def\co{\colon\thinspace} 

\def\commacat{\downarrow}

\newcommand{\Diff}{\mathrm{Diff}}

\newcommand{\id}{\mathrm{id}}

\newcommand{\Z}{\mathbb{Z}}

\newcommand{\Mod}{\mathrm{Mod}} 
\newcommand{\Sym}{\mathrm{Sym}} 
\newcommand{\hoMod}{\mathbb{L}\mathrm{Mod}} 

\DeclareMathOperator*{\colim}{colim} 
\DeclareMathOperator*{\hocolim}{hocolim} 

\newcommand{\Cat}{\mathscr{C}\hspace*{-1pt}\mathit{at}}
\newcommand{\Top}{\mathscr{T}\hspace*{-1pt}\mathit{op}}
\newcommand{\dgVect}{\mbox{dg-$\mathscr{V}\hspace*{-2pt}\mathit{ect}$}}

\newcommand{\Gext}{\mathscr{G}\hspace*{-1pt}\mathit{r}_{+}}
\newcommand{\G}{\mathscr{G}\hspace*{-1pt}\mathit{r}} 
 
\newcommand{\RG}{\mathscr{R}\mathit{ib}} 

\newcommand{\mobG}{\mathscr{M}\hspace*{-1pt}\mathit{\ddot{o}b}}

\newcommand{\Gred}{\G{}^\mathit{red}} 
\newcommand{\Gconn}{\G_\mathit{conn}} 
\newcommand{\Gtree}{\G_\mathit{tree}} 
\newcommand{\Gconnred}{\G_\mathit{conn}^\mathit{red}} 
\newcommand{\RGtree}{\RG_\mathit{tree}}

\newcommand{\mobGtree}{\mobG_\mathit{tree}}

\newcommand{\In}{\mathrm{In}} 
\newcommand{\HH}{\mathcal{H}} 

\newcommand{\MCG}{\mathcal{MCG}} 
\newcommand{\MCGoid}{\mathscr{S}} 
\newcommand{\NMCGoid}{\widetilde{\mathscr{S}}} 

\newcommand{\OS}{\mathcal{OS}} 
\newcommand{\NOS}{\widetilde{\mathcal{OS}}} 

\newcommand{\Hand}{\mathcal{H}\mathit{bdy}} 
\newcommand{\Ass}{\mathcal{A}\mathit{ss}} 
\newcommand{\InvAss}{\mathcal{A}\mathit{ss}^h} 

\begin{document}
\begin{abstract}
  The framed little 2-discs operad is homotopy equivalent to a cyclic
  operad.  We show that the derived modular envelope of this cyclic
  operad (i.e., the modular operad freely generated in a homotopy
  invariant sense) is homotopy equivalent to the modular operad made
  from classifying spaces of diffeomorphism groups of 3-dimensional
  handlebodies with marked discs on their boundaries.  A modification
  of the argument provides a new and elementary proof of K.~Costello's
  theorem that the derived modular envelope of the associative operad
  is homotopy equivalent to the ``open string'' modular operad made
  from moduli spaces of Riemann surfaces with marked intervals on the
  boundary.  Our technique also recovers a theorem of C.~Braun that
  the derived modular envelope of the cyclic operad that describes
  associative algebras with involution is homotopy equivalent to the
  modular operad made from moduli spaces of unoriented Klein surfaces
  with open string gluing.
\end{abstract}
\maketitle

\section{Introduction}

The purpose of this paper is to demonstrate how certain interesting
spaces in low-dimensional topology, namely classifying spaces of
mapping class groups of 3-dimensional handlebodies and oriented or
unoriented surfaces with boundary, can be built up from relatively
simple cyclic operads via the modular envelope construction.  Such
models lead to graph complexes that compute the cohomology of these
groups, generalising the well-known construction of the ribbon graph
complex that computes the cohomology of moduli spaces of Riemann
surfaces.

Cyclic operads were introduced by Getzler and Kapranov in
\cite{Getzler-cyclic}.  Roughly speaking, a cyclic operad is an operad
in which the roles of inputs and outputs are exchangeable.  A
prototypical example is the gluing of Deligne-Mumford-Knudsen
compactified moduli spaces of genus zero complex curves with marked
points.  Modular operads were introduced in \cite{Getzler-modular}.  A
modular operad is a cyclic operad together with self-gluings where an
input can be contracted with the output; a prototypical example is
given by the compactified moduli spaces of complex curves of arbitrary
genus.

A cyclic operad $\mathcal{O}$ generates a modular operad
$\Mod_!\mathcal{O}$ known as the \emph{modular envelope} of
$\mathcal{O}$.  The spaces of the modular envelope can be built as
certain colimits of spaces of $\mathcal{O}$ over categories of graphs.
This construction has a derived (homotopy invariant) version
$\hoMod_!\mathcal{O}$, in which the colimit is replaced by a homotopy
colimit.  See section \ref{mod-envelope-section} for details.

The framed little 2-discs operad $f\mathcal{D}_2$ was introduced by
Getzler in \cite{Getzler-BV} to describe homological conformal field
theories at genus zero.  Topologically, (group complete) algebras over
$f\mathcal{D}_2$ are (up to homotopy) 2-fold loops on spaces with a
circle action \cite{Salvatore-Wahl}, and the homology of
$f\mathcal{D}_2$ is the operad that describes Batalin-Vilkovisky
algebras.  Although $f\mathcal{D}_2$ is not a cyclic operad on the
nose, it is homotopy equivalent to a cyclic operad.  Various homotopy
equivalent cyclic models exist, such as the conformal balls operad of
\cite{Budney} or the compactified moduli spaces of rational pointed
curves with phase parameters from \cite [section 2.4]{KSM} (see also
\cite{cyclic-formality}).

In this paper we shall consider a cyclic model for $f\mathcal{D}_2$
made from diffeomorphism groups of 3-balls with marked 2-discs on
their boundary (the diffeomorphisms fix the discs point-wise).  This
cyclic model is naturally the genus zero part $\Hand_0$ of the modular
operad $\Hand$ made from classifying spaces of diffeomorphism groups
of 3-dimensional handlebodies with marked discs on their boundaries
and with the operad composition given by gluing along the marked
discs.  These operads will be constructed precisely in section
\ref{operad-defs-section}.  Our main theorem says roughly that the
handlebody modular operad $\Hand$ is freely generated (in the derived
sense) by its genus zero part.

\begin{thmA}\label{handlebody-theorem}
  There is a map of modular operads \[\hoMod_!\Hand_0 \to
  \Hand \] that is and isomorphism on $\pi_0$ and is a homotopy
  equivalence on all components except for the component corresponding
  to a solid torus with zero marked boundary discs.
\end{thmA}

The idea of the proof is the following.  A complete disc system in a
handlebody $K$ is a collection of 2-discs that partition $K$ into
genus zero pieces.  Contractibility of the space of complete disc
systems follows easily from the well-known contractibility of the
usual disc complex (the issue with the solid torus stems from
non-contractibility of the disc complex in this one exceptional case).
Thus the space of handlebodies is homotopy equivalent to the space of
handlebodies equipped with complete disc systems.  This space maps to
the space of graphs by sending a complete disc system to its dual
graph.  The modular envelope of the genus zero part $\Hand_0$ also
maps to the space of graphs.  We show the homotopy equivalence of the
theorem by a fibrewise comparison over the space of graphs.

Simplified versions of the above argument can in fact prove other
interesting theorems of this type.  It has long been known that ribbon
graphs give orbi-cell decompositions of moduli spaces of Riemann
surfaces.  In \cite{Costello1,Costello2} Costello proved the following
beautiful incarnation of this idea.  Let $\Ass$ denote the associative
operad.  It has the structure of a cyclic operad, and, up to homotopy,
it can be thought of as the framed little 1-discs operad.  It is
straightforward from the definitions to see that $\hoMod_!\Ass$ is the
modular operad of made from moduli spaces of metric ribbon graphs with
gluing at external legs.  Let $\OS$ denote the so-called ``open string
moduli space'' modular operad made from classifying spaces of mapping
class groups of oriented surfaces with marked intervals on their
boundaries and with the operad composition maps given by gluing
surfaces at marked intervals.  By Teichm\"uller theory, the spaces of
$\OS$ are homotopy equivalent to moduli spaces of Riemann surfaces
with boundary and marked points on the boundary, and this modular
operad governs \emph{open topological conformal field theories}.
\begin{thmA}\label{ribbon-graphs}
  There is a map of modular operads $\hoMod_!\Ass \to \OS$ that is a
  $\pi_0$ isomorphism and is a homotopy equivalence on all components
  except that of the annulus with zero marked boundary intervals.
\end{thmA}
Costello's proof uses the geometry of a certain compactification of
the moduli spaces of Riemann surfaces.  Our proof is instead based on
the well-known contractibility of the arc complex.

Our method can also be used to prove an unoriented version of
Theorem \ref{ribbon-graphs}.  Let $\NOS$ denote the ``unoriented
open string moduli space'' modular operad made from classifying spaces
of mapping class groups of unoriented surfaces (the surfaces need not
be orientable) with boundary and marked oriented intervals on the
boundary and with the operad compositions given by gluing of the
internals (compatibly with their orientations).  Next, consider the
`hermitian associative' operad $\InvAss$ in $\Top$ that describes
associative monoids with an involution $x \mapsto \overline{x}$ such
that $\overline{x\cdot y} = \overline{y} \cdot \overline{x}$.  This
cyclic operad can also be thought of (up to homotopy) as an unoriented
version of the framed little 1-discs, meaning that the framings need
not all have the same orientation.  The derived modular envelope of
$\InvAss$ is the modular operad of moduli spaces of M\"obius graphs,
which are an unoriented variant of the notion of ribbon graphs.  

\begin{thmA}\label{mobius-graphs}
  There is a map of modular operads $\hoMod_!\InvAss \to \NOS$ that is
  a $\pi_0$ isomorphism and is a homotopy equivalence on all
  components except those of the annulus or M\"obius band with zero
  marked boundary intervals.
\end{thmA}

Braun \cite{Braun} has recently independently introduced the
operads $\InvAss$ and $\NOS$ and given a very thorough study of their
properties.  In particular, he gives a proof of Theorem
\ref{mobius-graphs} following Costello's method.

\subsection{Graph homology and homotopy colimits}
The underlying spaces of the derived modular envelope of a cyclic operad
$\mathcal{O}$ are given as homotopy colimits of products of the spaces
of $\mathcal{O}$; this provides a Bousfield-Kan homology spectral
sequence that computes the homology of a modular envelope.

In \cite{cyclic-formality}, Salvatore and the author prove that the
cyclic operad $\Hand_0$ is formal over the reals, in the sense that
$C_*(\Hand_0;\mathbb{R})$ and $H_*(\Hand_0;\mathbb{R}) = \mathcal{BV}$
(the Batalin-Vilkovisky algebra cyclic operad) are quasi-isomorphic as
cyclic operads in $\dgVect$.  Together with Theorem
\ref{handlebody-theorem}, this implies that the Bousfield-Kan spectral
sequence for the real homology of $\Hand$ collapses at the $E^2$ page.
See section \ref{cohomology-section} for more details.

The information contained in the Bousfield-Kan spectral sequence for a
derived modular envelope can be reorganised in an interesting way
using \emph{graph homology}.  The construction of graph homology was
first introduced by Kontsevich in \cite{Kont-Feynman} and
\cite{Kont-symplectic}; it was subsequently generalised to arbitrary
cyclic operads in \cite{Conant-Vogtmann}.  The construction takes a
cyclic operad in chain complexes as input and produces a chain complex
of graphs labelled by the operad as output.  Kontsevich observed that
the graph homology of the Lie algebra cyclic operad
$\mathcal{L}\hspace*{-1pt}\mathit{ie}$ computes the cohomology of
(outer) automorphism groups of free groups, and the graph homology of
$\Ass$ computes the cohomology of moduli spaces of Riemann surfaces
with boundary.

In general, given a cyclic operad $\mathcal{O}$ in $\Top$, we shall
show in section \ref{cohomology-section} that the results of
\cite{Lazarev-Voronov} imply that the cohomology of the derived
modular envelope $\hoMod_!  \mathcal{O}$ is computed by the graph
homology complex of the cyclic operad $D(C_*\mathcal{O})$ that is the
dg-dual of the cyclic operad $C_*\mathcal{O}$.  It is well-known that
$\Ass$ is self-dual: $D(C_*\Ass) \simeq C_*\Ass$.  Braun \cite{Braun}
has shown that $\InvAss$ is also self-dual.  Hence one recovers the
statement that the graph cohomology of $\Ass$ and $\InvAss$ computes
the cohomology of classifying spaces of diffeomorphism groups of
oriented and unoriented surfaces respectively (except in the case of a
M\"obius band or annulus without marked boundary intervals). Moving
from dimension 2 to 3, formality of $\Hand_0$ and Theorem
\ref{handlebody-theorem} imply
\begin{thmA}
  $D(\mathcal{BV})$ graph cohomology computes the real cohomology of the
  classifying spaces of diffeomorphism groups of 3-dimensional
  oriented handlebodies (except in the case of a solid annulus with no
  marked discs).
\end{thmA}

The operad $D(\mathcal{BV})$ has been studied to some extent in \cite{Vallette}.


\section{Graphs, ribbon graphs, and M\"obius graphs}

\subsection{Graphs} 
The language of operads is built upon the language of graphs.  We
therefore begin by laying out some basic definitions of graphs.  For
us, a graph $\gamma$ will be a 1-dimensional finite CW complex; we
allow graphs with multiple connected components and with isolated
(0-valent) vertices, but we require that each component has at least
one vertex of valence $\neq 1$.  A \textbf{leg} is a univalent vertex,
an \textbf{external edge} is an edge that meets a leg, and an
\textbf{internal edge} is one for which neither end is univalent.  We
write $E_{\mathit{int}}(\gamma)$ for the set of internal edges of
$\gamma$, and $V(\gamma)$ for the set of non-leg vertices.  The
\textbf{rank} of a graph is the first Betti number of the graph, and a
graph is a \textbf{tree} if it is of rank 0.

Given two graphs $\gamma_1, \gamma_2$, a \textbf{subgraph collapse}
$\gamma_1 \to \gamma_2$ is the isotopy class of a surjective cellular
map that restricts to a bijection on external edges and such that the
inverse image of each vertex in $\gamma_2$ is a connected subgraph of
$\gamma_1$.  A \textbf{tree collapse} is a subgraph collapse whose
vertex preimages are trees.  We define the following categories of
graphs.  Let $\Gext$ be the category of graphs and subgraph collapses,
and let $\G \subset \Gext$ denote the subcategory of graphs and tree
collapses.  For reasons that will become clear when we define modular
envelopes, we will write
\[
\Mod\co \G \hookrightarrow \Gext
\]
for the inclusion functor.  Given a finite set $P$ we will write
$\G(P) \subset \Gext(P)$ for the versions of these categories in which
the graphs have $|P|$ legs equipped with a bijection to $P$ and the
morphisms respect these labellings.  Let $*_P \in \Gext$ denote the
corolla with one leg for each element of $P$. We shall often encounter
the comma category $(\Mod \commacat *_P)$ of graphs over $*_P$, which
is equivalent to the full subcategory $\Gconn(P)$ of $\G(P)$ on all
connected graphs.

\subsection{Graphs without bivalent vertices}
A vertex is \textbf{essential} if it is not bivalent, or it is
bivalent with two legs.  A graph is said to be \textbf{reduced} if all
vertices are essential.  We will write $\Gred \subset \G$ for the full
subcategory of reduced graphs, and likewise for each of the other
categories of graphs defined above.

There is a reduction functor
\[
R \co \G \to \Gred
\]
given by replacing each pair of edges meeting at an inessential
bivalent vertex with a single edge.  To define $R$ on morphisms, note
that any tree collapse factors as a series of (non-loop) edge
contractions $\pi_e\co \gamma \to \gamma/e$.  If both endpoints of $e$
are essential then there is a unique edge $\overline{e}$ in
$R(\gamma)$ such that $R(\gamma/e) = R(\gamma)/\overline{e}$ and we
define $R(\pi_e) = \pi_{\overline{e}}$.  If $e$ has an inessential
endpoint then $R(\gamma) = R(\gamma/e)$ and $R(\pi_e)$ is the
identity.  For $\gamma \in \Gred$, consider the comma category $(R
\commacat \gamma)$ and the full subcategory $(R \commacat \gamma)_0$
on those objects for which the map to $\gamma$ is an isomorphism in
$\Gred$ and there are no bivalent vertices over external edges of
$\gamma$, and let
\[
J\co (R \commacat \gamma)_0 \hookrightarrow (R \commacat \gamma)
\]
denote the inclusion functor.  For the proof of Theorem
\ref{handlebody-theorem} we shall need to relate homotopy colimits
over $(R \commacat \gamma)$ to homotopy colimits over this subcategory.

\begin{lemma}
  For any object $x \in (R \commacat \gamma)$, the comma category $(x
  \commacat J)$ has a final object.
\end{lemma}
\begin{proof}
  Given a graph $\tau \in \G$, there is a canonical isotopy class of
  homeomorphisms $\tau \cong R(\tau)$.  Now, for an object $x=(R(\tau)
  \stackrel{\alpha}{\to} \gamma)$ of $(R\commacat \gamma)$, let $E$
  denote the set of edges $e$ of $\tau$ which are mapped by $\alpha$
  to vertices of $\gamma$ (note that this condition is isotopy
  invariant and hence well defined).  The set $E$ is necessarily a
  forest in $\tau$.  Hence there is a tree collapse $\tau \to \tau/E$
  and an induced tree collapse $R(\tau) \to R(\tau/E)$.  By
  construction, there is a unique isomorphism $R(\tau/E) \to \gamma$ such
  that the diagram in $\Gred$,
\begin{equation*}
\begin{diagram}
\node{R(\tau)} \arrow{e} \arrow{s} \node{R(\tau/E)} \arrow{sw}\\
\node{\gamma} 
\end{diagram}
\end{equation*}
commutes.  This diagram gives the desired final object. 
\end{proof}

\begin{proposition}\label{bivalent-to-simplicial-prop}
For any functor $F\co (R\commacat \gamma) \to \Top$, $J$ induces a
homotopy equivalence,
\[
\hocolim_{(R\commacat \gamma)_0} F \stackrel{\simeq}{\longrightarrow}
\hocolim_{(R\commacat \gamma)} F.
\]
\end{proposition}
\begin{proof}
  We define a homotopy inverse to the map induced by $J$.  Sending
  $x\in (R\commacat \gamma)$ to a final object of $(x \commacat J)$
  defines a functor $W\co (R\commacat \gamma) \to (R\commacat
  \gamma)_0$ and a natural transformation from $\id_{(R\commacat
    \gamma)}$ to $W$.  This natural transformation induces a
  deformation retraction of $\hocolim_{(R\commacat \gamma)} F$ onto
  $\hocolim_{(R\commacat \gamma)_0} F$.
\end{proof}

Let $\Delta$ denote the usual category of finite nonempty ordered sets
and weakly order-preserving maps. Let $\Delta_{semi} \subset \Delta$
denote the subcategory of all injective maps.

\begin{proposition}\label{simplicial-equiv-prop}
There is an equivalence of categories
\[
(R\commacat \gamma)_0 \simeq \prod_{e \in E_\mathit{int}(\gamma)} \Delta_{semi}^{op}.
\]
\end{proposition}
\begin{proof} 
  If $R(\tau)$ is isomorphic to $\gamma$ then, in $\tau$, over each
  edge $e$ of $\gamma$ there is a set $S$ of edges meeting at
  inessential bivalent vertices.  If one chooses an orientation of $e$
  then $S$ inherits an ordering and hence determines an object of
  $\Delta_{semi}^{op}$.  Contracting a subset $E$ of the edges of $S$
  determines an injective order-preserving map $S \smallsetminus E
  \hookrightarrow S$.  This determines the equivalence of categories.
\end{proof}

\subsection{Ribbon graphs and M\"obius graphs} 

Here we shall recall some basic facts about ribbon graphs and a variant
known as M\"obius graphs.

\begin{definition}
A \textbf{ribbon graph} is a graph equipped with a cyclic ordering of
the half-edges incident at each vertex.
\end{definition}

A ribbon graph $\gamma$ can be canonically thickened to an oriented
surface $S(\gamma)$ with boundary; legs of the ribbon graph correspond
to marked intervals on the boundary of the surface.
\begin{center}
\includegraphics{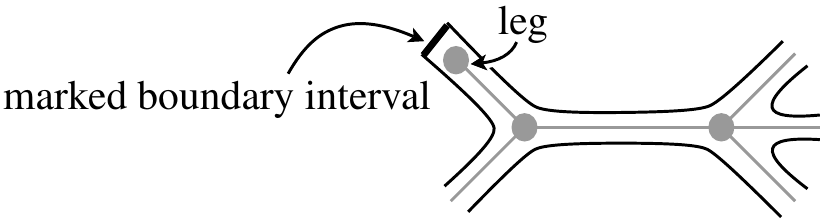}
\end{center}

Given an internal edge $e$ in a ribbon graph $\gamma$, one sees that
the graph $\gamma/e$ formed by contracting $e$ inherits a ribbon
structure from $\gamma$.  We thus let $\RG$ denote the category of
ribbon graphs and tree collapses that respect the ribbon structures.
Let $\RG(P)$ denote the category of ribbon graphs with legs labelled
by $P$.  

Let $\Sym\co \RG \to \G$ denote the functor that forgets the ribbon
structure.  The reason for the name will become clear in section
\ref{sym-section} when we discuss the cyclic operad generated by a
non-$\Sigma$ cyclic operad.  Observe that $(\Sym \commacat *_P)$ is
equivalent to the full subcategory $\RGtree(P) \subset \RG(P)$ of
ribbon trees.  This subcategory has several connected components, and
each component has a final object given by a ribbon corolla.

\subsection{M\"obius graphs}
We now define M\"obius graphs, which are a slight variation on ribbon
graphs and give models for nonorientable surfaces.  A
\textbf{pre-M\"obius structure} on a graph consists of:
\begin{enumerate}
\item a cyclic ordering on the half-edges incident at each vertex,
\item a labelling of the edges by elements of $\Z/2$.
\end{enumerate}
There is an equivalence relation on the set of pre-M\"obius structures on
a given graph $\gamma$ generated by the following operation: reverse
the cyclic order on the half-edges at a vertex and reverse the $\Z/2$
labels on all non-loop edges incident at that vertex.  
\begin{definition}
  A \textbf{M\"obius graph} is a graph equipped with an equivalence
  class of pre-M\"obius structures.
\end{definition}

A pre-M\"obius structure on a graph $\gamma$ determines a canonical
thickening to a (not necessarily orientable) surface $S(\gamma)$ with
marked directed intervals on the boundary corresponding to the legs;
the construction is the same as for ribbon graphs, except that an edge
labelled by the nontrivial element $1\in \Z/2$ now corresponds to a
strip that is glued in with a half twist, and a leg at the end of an
external edge labelled $1$ now gives a marked boundary interval
oriented opposite to the cyclic order at the other end of the edge, as
illustrated below.
\begin{center}
\includegraphics{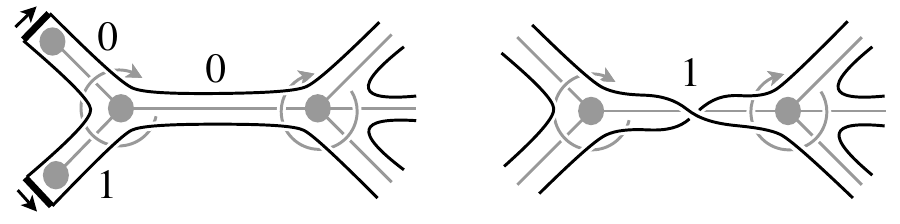}
\end{center}
For example, the M\"obius graph consisting of a single edge with both
ends meeting at a bivalent vertex thickens to the annulus $S^1\times
I$ if the edge is labelled by $0$, and it thickens to a M\"obius band
if the edge is labelled by $1$.  It is straightforward to see that if
two pre-M\"obius structures are equivalent then the corresponding
thickenings will be canonically homeomorphic.  Thus a M\"obius graph
has a well-defined thickening.

As with ribbon graphs, if the source of a tree collapse has a M\"obius
structure then the target inherits a well-defined structure and hence
there is a category $\mobG$ of M\"obius graphs and tree collapses that
respect the M\"obius structure.  Let $\mobG(P)$ denote the category of
M\"obius graphs with $P$ legs.  

As for ribbon graphs, let $\Sym'\co \mobG \to \G$ denote the functor
that forgets the M\"obius structure.  Observe that $(\Sym' \commacat
*_P)$ is equivalent to the full subcategory $\mobGtree(P) \subset
\mobG(P)$ of M\"obius trees.  This subcategory has several connected
components, and each component has a M\"obius corolla as final object.

\section{Definition and homotopy theory of cyclic and modular operads}
\label{operad-defs-section}

\subsection{A convenient definition for various flavours of operads}
Our perspective on cyclic and modular operads is heavily inspired by
that of Costello in \cite{Costello1}.  Below we shall give definitions of
cyclic and modular operads that are equivalent to any of the usual
definitions but are slightly more convenient for the homotopy
theory that we will be doing.

Forgetting the labelling of the legs on a graph gives a functor $\G(P)
\to \G$.  Given two distinct elements $i,j \in P$ there is a
\emph{gluing functor}
\[
glue_{i,j}\co \G(P) \to \G(P \smallsetminus \{i,j\}) 
\]
defined on objects by gluing the $i^{th}$ leg to the $j^{th}$ leg and
replacing the resulting pair of edges meeting at a bivalent vertex with a
single edge.
\begin{center}
\includegraphics{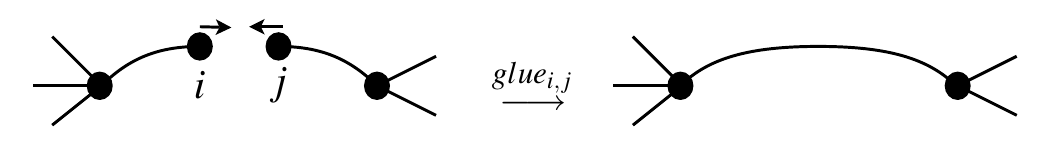}
\end{center}
Disjoint union of graphs makes $\G$ into a symmetric monoidal
category.  The monoidal product and gluing functors are also defined
for $\Gext$, ribbon graphs, and M\"obius graphs.  

\begin{definition}\label{operad-def}
  Let $(\mathscr{C},\otimes)$ be a symmetric monoidal category.
\begin{enumerate}
\item A \textbf{cyclic operad} in $\mathscr{C}$ is a symmetric monoidal functor
  $\mathcal{O}\co \G \to \mathscr{C}$ that commutes with all graph
  gluing functors (up to natural isomorphism),
  \[
  \begin{diagram}
    \node{\G(P)} \arrow{s,r}{glue_{i,j}} \arrow{e}
    \node{\G} \arrow{se,t}{\mathcal{O}} \\
    \node{\G(P\smallsetminus \{i,j\})} \arrow{e} \node{\G} \arrow{e,t}{\mathcal{O}}
    \node{\mathscr{C}.}
  \end{diagram}
  \]

\item A \textbf{modular operad} in $\mathscr{C}$ is a symmetric monoidal functor
  $\mathcal{O}\co \Gext \to \mathscr{C}$ that commutes with all graph
  gluing functors.

\item A \textbf{non-$\Sigma$ cyclic operad} in $\mathscr{C}$ is a symmetric
  monoidal functor $\mathcal{O}\co \RG \to \mathscr{C}$ that commutes
  with all gluing functors.

\item A \textbf{M\"obius cyclic operad} in $\mathscr{C}$ is a symmetric
  monoidal functor $\mathcal{O}\co \mobG \to \mathscr{C}$ that commutes
  with all gluing functors.
\end{enumerate}
\end{definition}

\begin{remark}
  Operads are often required to have a unit in the space
  $\mathcal{O}(*_2)$ associated with a bivalent corolla.  However, in
  this paper we do not impose this requirement.
\end{remark}
We sketch how this definition is equivalent to other definitions found
in the literature.  Observe that commuting with the gluing functors
implies that the values of a cyclic or modular operad $\mathcal{O}$
are determined by its values on disjoint unions of corollas.  Since it
is symmetric monoidal, its values are in fact determined by what it
does on single corollas.  Any subgraph collapse can be factored as a
sequence of edge contractions (non-loop edge contractions in the case
of a tree collapse), so the functor $\mathcal{O}$ is completely
determined by its values on corollas and what it does when contracting
the edge in a gluing of two corollas (and self-gluings of a single
corolla in the case of a modular operad).  Thus a cyclic operad in our
sense is equivalent to giving, for any finite set $P$, an object
$\mathcal{O}(P)$ with an action of $\Sigma_P$ and for any $i\in P,
j\in Q$ an operadic composition map
\[
_i\circ_j\co \mathcal{O}(P) \otimes \mathcal{O}(Q) \to \mathcal{O}(P\sqcup
Q \smallsetminus \{i,j\}),
\]
satisfying certain compatibility conditions.  A modular operad is a
functor for which contracting loops is allowed, so it additionally
comes with a composition map
\[
\circ_{ij}\co \mathcal{O}(P) \to \mathcal{O}(P\smallsetminus \{i,j\})
\]
for any $i,j\in P$.  We shall sometimes find it convenient to
informally define a cyclic or modular operad in these terms.

A morphism of cyclic or modular operads $\mathcal{O} \to \mathcal{P}$
is natural transformation of functors.  For operads in the category
$\Top$ of spaces, a morphism is said to be a homotopy equivalence if
it is pointwise a homotopy equivalence of spaces.  Similarly, a
morphism of operads in $\dgVect$ is said to be a quasi-isomorphism if
it is a quasi-isomorphism pointwise.

\subsection{Connected graphs as a modular operad}
Let $\Cat$ denote the category of all small categories with Cartesian
product as symmetric monoidal product.  An important example of a
modular operad is provided by categories of graphs themselves.  The
rule,
\[
P \mapsto \Gconn(P),
\]
constitutes a modular operad in $\Cat$, where the composition maps
$_i\circ_j$ and $\circ_{ij}$ are defined by gluing the $i$ leg to the
$j$ leg and forgetting the resulting bivalent vertex.  In terms of
Definition \ref{operad-def}, as a functor $\Gext \to \Cat$, this
modular operad sends $\gamma$ to $(\Mod \commacat \gamma)$.

The full subcategories $\Gtree(P) \subset \Gconn(P)$ of trees
constitute a cyclic operad but \emph{not} a modular operad.
Similarly, the categories $\RGtree(P) \simeq (\Sym \commacat *_P)$ and
$\mobGtree(P) \simeq (\Sym' \commacat *_P)$ of ribbon trees and
M\"obius trees respectively, constitute cyclic operads as well.

\subsection{Left Kan extensions and homotopy left Kan extensions}
Let $\mathscr{A}$, $\mathscr{B}$ and $\mathscr{C}$ be categories with
$\mathscr{C}$ cocomplete.  Consider functors 
\[
\begin{diagram}
\node{\mathscr{A}} \arrow{e,t}{F} \arrow{s,l}{G} \node{\mathscr{C}} \\
\node{\mathscr{B}.}
\end{diagram}
\]
Recall that the \textbf{left Kan extension} of $F$ along $G$ is a functor $G_! F\co
\mathscr{B} \to \mathscr{C}$ defined on objects by the colimit
\[
G_! F (b) = \colim_{(G \commacat b)} F\circ j_b,
\]
where $(G\commacat b)$ is the comma category of objects in
$\mathscr{C}$ over $b$ and $j_b\co (G \commacat b) \to \mathscr{C}$
forgets the morphism to $b$ (to simplify the notation we will often
omit writing $j_b$).  Left Kan extensions possess a universal property:  
the functor $G_!F$ comes with a natural transformation
\[
F \Rightarrow G_! F \circ P 
\]
that is initial among natural transformations from $F$ to functors
factoring through $P$.

If $\mathscr{C}$ is a Quillen model category (such as the category
$\Top$ of topological spaces) then there is a derived (homotopy
invariant) version known as the \textbf{homotopy left Kan extension}
$\mathbb{L}G_!F$; it is given by the formula
\[
\mathbb{L}G_! F (b) = \hocolim_{(G \commacat b)} F\circ j_b.
\]
There is a homotopy coherent version of the above universal property
for homotopy left Kan extensions.

Note that there is a ``Fubini theorem'' for ordinary and homotopy colimits,
\[
\colim_{\mathscr{A}} F \cong \colim_{\mathscr{B}} G_! F \mbox{\:\:\: and \:\:\:}
\hocolim_{\mathscr{A}} F \stackrel{\simeq}{\to} \hocolim_{\mathscr{B}} \mathbb{L}G_! F.
\]

\subsection{Homotopy colimits, diagrams in $\Cat$ and Thomason's Theorem}
At several points we shall be taking homotopy colimits of diagrams in
$\Top$ obtained from diagrams in $\Cat$ by applying the classifying
space functor $B$ (i.e. geometric realisation of the nerve) pointwise.
Here we briefly recall a couple of useful tools for this situation.

Given a functor $F\co \mathscr{C} \to \Cat$, the \textbf{Grothendieck
  construction} on $F$, denoted $\mathscr{C}\int F$ is the category in
which objects are pairs $(x\in \mathscr{C}, y\in F(x))$, and a
morphism $(x,y) \to (x',y')$ consists of an arrow $f\in
\hom_\mathscr{C}(x,x')$ and an arrow $g\in \hom_{F(x')}(f_*y,y')$.
By Thomason's Theorem \cite[Theorem 1.2]{Thomason}, there is a natural
homotopy equivalence,
\[
\hocolim_{\mathscr{C}} BF \stackrel{\simeq}{\longrightarrow} B\left( \mathscr{C} \int F \right).
\]

As a special case, if $\mathscr{C}$ is actually a group $G$ (a
category with a single object $*$ and all arrows invertible), then
$BF(*)$ is a space with a $G$ action, and $B(G\int \mathscr{F})$ is
homotopy equivalent to the homotopy quotient $(B\mathscr{F})_{hG}$.

If $\mathscr{C} = \Delta^{op}_{\mathit{semi}}$ then $F$ is a
semi-simplicial category, $BF$ is a semi-simplicial space, and
$B(\Delta^{op}_{\mathit{semi}} \int F) \simeq \hocolim BF$ is
equivalent to the geometric realisation of this semi-simplicial space.

\subsection{From non-$\Sigma$ or M\"obius cyclic operads to cyclic operads}\label{sym-section}
Given a non-$\Sigma$ cyclic operad $\mathcal{O}\co \RG \to
\mathscr{C}$, left Kan extension along $\Sym\co \RG \to \G$ gives a
cyclic operad $\Sym_!\mathcal{O}$ called the \textbf{symmetrisation} of
$\mathcal{O}$.  The fact that this is indeed a cyclic operad follows
easily from the properties of $\mathcal{O}$ and the fact that $\gamma
\mapsto (\Sym \commacat \gamma) \simeq \prod_{v\in V(\gamma)}
\RGtree(*_v)$ is a cyclic operad (where $*_v$ is the corolla in
$\gamma$ at $v$).  Explicitly, if $*_P$ is a corolla with set of legs
$P$, then
\[
\Sym_!\mathcal{O}(*_P) = \colim_{\RGtree(P)} \mathcal{O}.
\]

Likewise, if $\mathcal{O}\co \mobG \to \mathscr{C}$ is a M\"obius
cyclic operad then it has a symmetrisation defined by left Kan
extension along $\Sym'\co \mobG \to \G$.  In this case, the
symmetrisation sends $*_P$ to $\colim_{\mobGtree(P)} \mathcal{O}$.

For non-$\Sigma$ or M\"obius cyclic operads in a model category such
as $\Top$, there is a derived symmetrisation given by homotopy left
Kan extension. However, since each component of the categories $(\Sym
\commacat *_P)$ and $(\Sym' \commacat *_P)$ has a final object, the
derived symmetrisation is homotopy equivalent to the ordinary symmetrisation.

\subsection{Modular envelopes: from cyclic operads to modular
  operads}\label{mod-envelope-section}
The \textbf{modular envelope} of a cyclic operad $\mathcal{O}$ is the
modular operad given by the left Kan extension $\Mod_!\mathcal{O}$
along the functor $\Mod\co \G \hookrightarrow \Gext$.  This should be
thought of as the modular operad freely generated by $\mathcal{O}$.
As with symmetrisations, one can easily check that this indeed defines
a modular operad.  Explicitly, the modular envelope sends the $P$-legged
corolla $*_P$ to
\[
\Mod_!\mathcal{O}(*_P) = \colim_{\Gconn(P)} \mathcal{O}.
\]
The modular operad composition morphisms $_i\circ_j$ and $\circ_{ij}$
are induced by the corresponding composition morphisms on the categories
$(\Mod \commacat *_P) \simeq \Gconn(P)$.

When $\mathcal{O}$ is a cyclic operad in a model category such as
$\Top$ then there is a derived (homotopy invariant) version of the
above construction, known as the \textbf{derived modular envelope},
$\hoMod_!\mathcal{O}$.  Its value on a corolla $*_P$ is
\[
\hoMod_!\mathcal{O}(*_P) = \hocolim_{\Gconn(P)} \mathcal{O}.
\]

Unlike for symmetrisations of non-$\Sigma$ or M\"obius cyclic operads,
the derived modular envelope of a cyclic operad can often be
very different from the ordinary version of the modular envelope.  Our
focus in this paper is on derived modular envelopes.

\subsection{Genus grading in modular operads}
Often it is the case that a modular operad $\mathcal{O}$ has an
internal grading by ``genus''.  This means that the object
$\mathcal{O}(*_P)$ is a coproduct
\[
\mathcal{O}(P) = \coprod_{g\geq 0} \mathcal{O}_g(p),
\]
the composition morphism $_i\circ_j$ is degree zero in the sense that
it sends $\mathcal{O}_g(P) \times \mathcal{O}_h(Q)$ to
$\mathcal{O}_{g+h}(P \sqcup Q \smallsetminus \{i,j\})$ and the loop
contraction map $\circ_{ij}\co \mathcal{O}(P) \to
\mathcal{O}(P\smallsetminus \{i,j\})$ is degree one, meaning that it
sends the $g$ component to the $g+1$ component.

Note that the modular envelope and derived modular envelope of a
cyclic operad always have a genus grading.  This is
because the modular operad of categories of connected graphs, $*_P
\mapsto \Gconn(P)$ has a genus grading, with the genus $g$ piece
consisting of the component of graphs of rank exactly $g$.

If $\mathcal{O}$ is a modular operad with a genus grading then the
genus zero part $\mathcal{O}_0$ forms a cyclic operad.

\section{Various cyclic and modular operads}

\subsection{The cyclic operads $\Ass$ and $\InvAss$}

Let $\Ass$ denote the associative cyclic operad in $\Top$.  It is
constructed as the symmetrisation of the non-$\Sigma$ cyclic operad
that is the constant functor sending any ribbon graph to a single
point.  Thus $\Ass (*_P)$ is a discrete set of points with one point
for each ribbon corolla with $P$ legs, or equivalently, one point for
each cyclic ordering of $P$.

The hermitian associative operad, $\InvAss$, is defined as the
symmetrisation of the the M\"obius cyclic operad that sends any graph
to a point. Thus $\InvAss (*_P)$ can be identified with the set of
isomorphism classes of M\"obius corollas with $P$ legs.

\subsection{The open string moduli space modular operads $\OS$ and $\NOS$}
Here we construct modular operads from mapping class groups of
oriented and unoriented surfaces with marked intervals on their
boundary and ``open string'' gluing.

Let $S$ be a compact connected surface with nonempty boundary and
equipped with $n\geq 0$ intervals marked on its boundary and
parametrised by $[0,1]$.  We do not require that $S$ is orientable,
but if is equipped with an orientation then we require that the
intervals be parametrised compatibly with the orientation.  Such a
surface is called \textbf{admissible} if it is neither (1) a disc with
strictly fewer than 2 marked boundary intervals, nor (2) an annulus or
M\"obius band with zero marked boundary intervals.  (We will be forced
to consider the annulus and M\"obius band because, although they are
not admissible, they can be formed by gluing together admissible
surfaces.)

Let $\Diff(S)$ denote the group of diffeomorphisms that fix the marked
intervals pointwise and preserve the orientation if $S$ is equipped
with an orientation.  By \cite{Earle-Eells}, \cite{Earle-Schatz} and
\cite{Gramain}, if $S$ is admissible then the components of $\Diff(S)$
are contractible. The mapping class group $\MCG(S)$ is the group
$\pi_0 \Diff(S)$ of isotopy classes of diffeomorphisms.  The identity
component of the diffeomorphism group of an annulus or M\"obius band
is homotopy equivalent to $S^1$.

To form a modular operad we replace the mapping class groups with
equivalent \emph{groupoids}.  This construction is a variation of the
construction used in \cite{Tillmann-surface-operad}.  Given a finite
set $P$ (possibly empty), let $\MCGoid(P)$ denote the groupoid in
which:
\begin{itemize}
\item Objects are admissible oriented surfaces with marked intervals
  in bijection with $P$.  If $P=\emptyset$ then we also allow an annulus.
\item Morphisms are isotopy classes of orientation-preserving
  diffeomorphisms that respect the marked intervals.
\end{itemize}
Similarly, let $\NMCGoid(P)$ denote the unoriented version of the
above category (when $P=\emptyset$ then we allow both the annulus and
the M\"obius band as objects).  Sending the corolla $*_P$ to
$\MCGoid(P)$ or $\NMCGoid(P)$ endows these two families of groupoids
each with the structure of a modular operad in $\Cat$ with the
composition morphisms induced by gluing pairs of marked intervals
together via a direction-reversing identification $[0,1] \to [0,1]$.

For $P\neq \emptyset$, the components of the space $B\MCGoid(P)$ (or
$B\NMCGoid(P)$) have the homotopy type of $B\MCG(S)$ for various
admissible surfaces $S$.  Note that this statement fails when $P$ is
empty for the components corresponding to the annulus or M\"obius
band.

\begin{definition}
  The oriented open string modular operad $\OS$ is given by $\OS(*_P)
  = B\MCGoid(P)$.  Similarly, the unoriented open string modular
  operad $\NOS$ is given by $\NOS(*_P) = B\NMCGoid(P).$
\end{definition}

The modular operads $\OS$ and $\NOS$ have genus gradings; however,
\emph{these gradings do not correspond to the genera of the surfaces.}
Rather, the genus $g$ part of $\OS(*_P)$ is the classifying space of
the subgroupoid $\MCGoid_g(P)$ of surfaces with first Betti number $b_1=g$.

\begin{proposition}\label{os-gzero-part}
  There are homotopy equivalences of cyclic operads, 
\[
\OS_0 \simeq \Ass \mbox{\:\:\:\: and \:\:\:\:} \NOS_0 \simeq \InvAss.
\]
\end{proposition}
\begin{proof}
  By definition of the cyclic associative operad, 
  \[\Ass(*_P) = \Sym_! pt (*_P) \simeq \mathbb{L}\Sym_!
  pt (*_P) = B(\Sym \commacat *_P),
  \]
  and $(\Sym \commacat *_P)$ is the groupoid of ribbon corollas with
  set of legs $P$.  Sending a ribbon corolla to its thickening defines
  an equivalence of groupoids $(\Sym \commacat *_P) \simeq
  \MCGoid_0(P)$.  This gives a homotopy equivalence $\Ass(*_P) \simeq
  \OS_0(*_P)$ that is easily seen to be compatible with the cyclic
  operad structures.  The argument of $\InvAss$ is the same.
\end{proof}

\subsection{The handlebody modular operad $\Hand$}

Let $K=K_{g,n}$ be a 3-dimensional compact connected oriented
handlebody of genus $g\geq 0$ with $n\geq 0$ disjoint 2-discs marked
on its boundary.  The boundary $\partial K$ is a closed surface of
genus $g$ with $n$ embedded discs.  We call $K$ \textbf{admissible} if
$(g,n) \neq (0,0), (0,1), (1,0)$.  (Note that excluding the case of
genus zero with one marked disk is not strictly necessary, but it will
make the exposition somewhat clearer.)  We will be forced to consider
one exceptional handlebody, namely the solid torus with zero marked
boundary discs; this handlebody is not admissible but it can arise
when gluing boundary discs on admissible handlebodies.

If $K$ is admissible then the components of the diffeomorphism group
of $K$ (fixing the marked discs pointwise) are contractible
\cite{Hatcher}; when $K$ is a solid torus with no marked boundary
discs then the identity component of the diffeomorphism group is
homotopy equivalent to $S^1\times S^1$.  Let $\HH(K)$ denote the group
of isotopy classes of diffeomorphisms of $K$ that fix the discs
pointwise, and let $\MCG(\partial K)$ denote the group of isotopy
classes of diffeomorphism of $\partial K$ fixing the discs pointwise.
Restriction of diffeomorphisms to the boundary of $K$ is injective on
the level of isotopy classes, so one can regard $\HH(K)$ as a subgroup
$\MCG(\partial K)$ \cite{Hatcher}.  Thus $\HH(K)$ is often called the
\textbf{handlebody subgroup} of the mapping class group.  Note that
for positive genus the handlebody subgroup depends on the choice of
how to realise a given surface as the boundary of a handlebody, but
any two choices determine conjugate subgroups.  However, when the
genus is zero then the handlebody subgroup is equal to the entire
mapping class group.

We will now construct a modular operad $\Hand$ from classifying
spaces of handlebody groups with the compositions given by gluing
marked discs together.  However, as in the previous section we must
replace the groups with equivalent groupoids.  Given a finite set $P$
(possibly empty), let $\mathscr{H}(P)$ denote the groupoid in which:
\begin{itemize}
\item Objects are 3-dimensional compact connected oriented admissible
  handlebodies with an identification of the marked boundary discs
  with $\coprod_P D^2$.  If $P= \emptyset$ then we also allow the
  (non-admissible) solid torus as an object.
\item Morphisms are isotopy classes of orientation-preserving
  diffeomorphisms that respect the identification of the marked
  boundary discs.
\end{itemize}
Fix an orientation reversing diffeomorphism $D^2 \to D^2$.  Given
$i\in P$ and $j\in Q$, gluing the $i$ disc to the $j$ disc using the
fixed diffeomorphism defines a functor $_i\circ_j\co
\mathscr{H}(P)\times \mathscr{H}(Q) \to \mathscr{H}(P\sqcup Q
\smallsetminus \{i,j\})$.  The self-gluing functor $\circ_{ij}$ is
defined similarly.

\begin{definition}
  The handlebody modular operad, $\Hand$, is the functor $\Gext \to
  \Top$ defined on objects by sending a corolla $*_P$ with set of legs
  $P$ to the space $B\mathscr{H}(P)$ and defined on morphisms by gluing
  handlebodies together at the marked discs using the above gluing
  functors.
\end{definition}

The handlebody modular operad clearly has a genus grading.
Restricting to the genus zero part of the handlebody groupoids gives a
cyclic operad $\Hand_0(*_P) = B\mathscr{H}_0(P)$.

\begin{proposition}
The framed little 2-discs operad $f\mathcal{D}_2$ is homotopy
equivalent to $\Hand_0$.
\end{proposition}
\begin{proof}
  Let $\Pi\co \Top \to \Cat$ denote the fundamental goupoid functor.
  Since the spaces of $f\mathcal{D}_2$ are all $K(\pi,1)$s (the groups
  $\pi$ here are ribbon braid groups), there is a homotopy equivalence
  of operads 
  \[
  B\Pi(f\mathcal{D}_2) \simeq f\mathcal{D}_2.
  \]
  We define a morphism of operads $\Pi(f\mathcal{D}_2) \to
  \mathscr{H}_0$ in $\Cat$ as follows.  Write $f\mathcal{D}_2(n)$ for
  the space of $n$ framed little discs in a standard disc.  An object
  of $\Pi(f\mathcal{D}_2(n))$ is a configuration of $n$ parametrised
  discs in a standard disc.  Such a configuration determines an object
  of $\mathscr{H}_0(\mathit{Legs}(*_{n+1}))$ by thinking of the
  standard disc as the northern hemisphere on the boundary of a 3-ball
  and choosing a standard parametrisation of the southern hemisphere.
  It is not hard to see that a path in $f\mathcal{D}_2(n)$ determines
  an isotopy class of diffeomorphisms between the 3-balls
  corresponding to the endpoints.  Hence we have the desired functor.
  Moreover, it is not hard to see that this functor is an equivalence
  of categories and is compatible with the operadic compositions.
\end{proof}

\section{Proofs of Theorems \ref{ribbon-graphs} and \ref{mobius-graphs}}

The argument we present for Theorems \ref{ribbon-graphs} and
\ref{mobius-graphs} is a simplified version of the proof of Theorem
\ref{handlebody-theorem} that we present in the next section.

Regarding $\OS$ as a cyclic operad for a moment, there is an
inclusion of cyclic operads, $\OS_0 \hookrightarrow \OS$.  Since the
target is actually a modular operad, the universal property of the
homotopy left Kan extension provides a map of modular operads,
\[
\Phi\co \hoMod_!\OS_0 \to \OS.
\]
We will show that this map is an equivalence.  Similarly, in the
unoriented setting, there is a map $\widetilde{\Phi} \co
\hoMod_!\NOS_0 \to \NOS$ that we will show to be a homotopy
equivalence.  It suffices to show that these maps are homotopy
equivalences when evaluated on corollas $*_p$.

Let us first describe the above maps more explicitly.  By Thomason's
Theorem,
\[
\hoMod_!\OS_0(*_P) \simeq B\left(\Gconn(P) \int \MCGoid_0(P) \right), 
\]
An object of the category on the right is a graph $\tau \in \Gconn(P)$
and a labelling of each internal vertex $v$ by a surface $S_v\in
\MCGoid_0(\In(v))$ that is a disc with one marked boundary interval for
each half-edge incident at $v$ (note that this is in essence just the
datum of a ribbon graph structure on $\tau$).  Gluing these labels
together according to the edges of $\tau$ gives an object $S \in
\MCGoid(P)$ and thus defines a functor
\[
\phi\co \Gconn(P) \int \MCGoid_0(P) \to \MCGoid(P).
\]
After taking classifying spaces, this map gives a model for $\Phi$ up
to homotopy.

For the proof, we will replace $\MCGoid(P)$ with a homotopy equivalent
category, lift $\phi$ to this new category, and show that the lift is
actually an equivalence of categories.

\subsection{Arc systems}

Let $S$ be a surface with $n$ marked boundary intervals.  An
\textbf{arc} in $S$ is an unparametrised curve in $S$ that meets
$\partial S$ only at its endpoints, which are required to be disjoint
from any of the marked boundary intervals.  Given an arc $C$, we may
cut $S$ along $C$ to obtain a new (possibly disconnected surface) with
$n+2$ marked boundary intervals coming from the original $n$ marked
intervals plus an additional marked interval on either side of the
cut, as shown below.
\begin{center}
\includegraphics{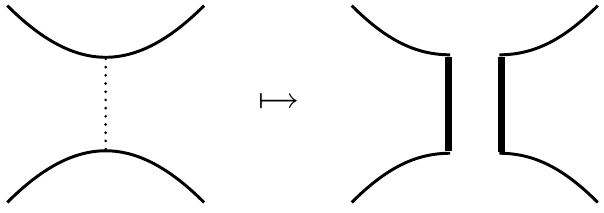}
\end{center} 

\begin{definition}
  An \textbf{arc system} $\alpha$ in $S$ is a nonempty set of disjoint
  arcs $C_i$ in $S$ such that if one cuts $S$ along all of the $C_i$
  then each component of the result is an admissible surface.  An arc
  system $\alpha$ in $S$ is said to be \textbf{complete} if each
  component of the result of cutting $S$ along the arcs in $\alpha$ is
  an admissible disc.
\end{definition}

Isotopy classes of arc systems form a category in which there is one
arrow $\alpha \to \beta$ for each way of forgetting some of the arcs
of $\alpha$ and identifying the isotopy class of the remaining arcs
with $\beta$.  Let $\mathscr{A}(S)$ denote the category of isotopy
classes of arc systems in $S$, and let $\mathscr{A}_0(S)$ denote the
full subcategory of complete arc systems.

\begin{proposition}\label{arc-contractibility}
  For any admissible surface $S$, the spaces $B\mathscr{A}(S)$ and
  $B\mathscr{A}_0(S)$ are contractible.
\end{proposition}
\begin{proof}
  The argument of \cite{Hatcher-triang} shows that $\mathscr{A}$ is
  contractible (note that our arc systems are slightly different from
  the arc systems considered by Hatcher since we allow parallel arcs,
  but it is straightforward to deal with this difference).  We will
  show that the inclusion $\iota\co \mathscr{A}_0 \hookrightarrow
  \mathscr{A}$ is a homotopy equivalence on classifying spaces by
  Quillen's Theorem A \cite{Quillen} and induction on the complexity
  of $S$.  Let $b_1$ be the first Betti number of $S$ and let $n$
  denote the number of boundary components.  Order the pairs $(b_1,
  n)$ lexicographically.  If $S$ is a disc ($b_1=0$) then
  $\mathscr{A}_0(S) = \mathscr{A}(S)$.  Now suppose that
  $B\mathscr{A}_0(T)\simeq pt$ for all admissible surfaces $T$ with
  $(b_1(T),n(T)) < (b_1(S),n(S))$.  Given any arc system $\alpha \in
  \mathscr{A}(S)$, let $\{S_i\}_{i\in I}$ denote the set of connected
  admissible surfaces resulting from cutting $S$ along the arcs in
  $\alpha$.  Observe that each component has $(b_1(S_i), n(S_i)) <
  (b_1(S), n(S))$.  There is an equivalence of categories,
  \[
  (\iota \commacat \alpha) \simeq \prod_{i\in I} \mathscr{A}_0(S_i),
  \]
  and hence, by the inductive hypothesis, $B(\iota \commacat \alpha)$
  is contractible.  Thus $\iota$ induces a homotopy equivalence of
  classifying spaces.
\end{proof}
\begin{remark}
  In the exceptional cases of an annulus or M\"obius band with zero
  marked boundary intervals, the space of complete arc systems is not
  contractible.
\end{remark}

We can regard $S \mapsto \mathscr{A}_0(S)$ as a functor $\MCGoid \to
\Cat$ or $\NMCGoid \to \Cat$ (depending on whether $S$ is equipped
with an orientation or not), and then form the Grothendieck
constructions $\MCGoid(P) \int \mathscr{A}_0$ and $\NMCGoid(P) \int
\mathscr{A}_0$.  By Proposition \ref{arc-contractibility} and
Thomason's Theorem, we have:

\begin{lemma}
The projections
\begin{align*}
\MCGoid(P) \int \mathscr{A}_0 \to \MCGoid(P) \mbox{\:\:\:\: and \:\:\:\:}
\NMCGoid(P) \int \mathscr{A}_0 \to \NMCGoid(P)
\end{align*}
are homotopy equivalences, except on the components corresponding to
an annulus or M\"obius band if $P=\emptyset$.
\end{lemma}

Now observe that $\phi$ lifts to a functor $\phi'\co \Gconn(P) \int
\MCGoid_0 \to \MCGoid(P) \int \mathscr{A}_0$ by sending an
$\MCGoid_0$-labelled graph to the surface obtained by gluing the
labels and the complete arc system with one arc for each pair of
boundary intervals that were glued, as illustrated below.
\begin{center}
\includegraphics{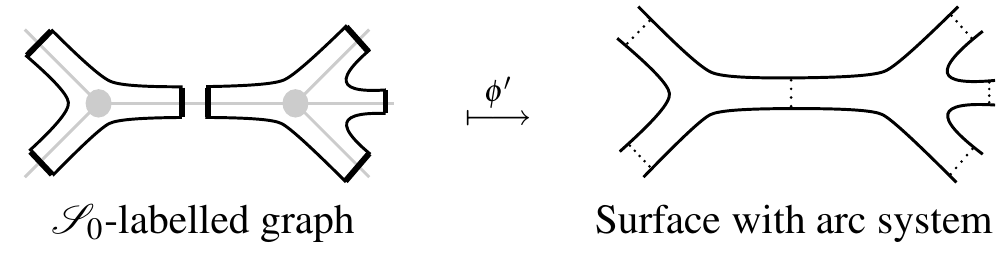}
\end{center}
Likewise, $\widetilde{\phi}$ lifts to $\widetilde{\phi}' \co
\Gconn(P) \int \NMCGoid_0 \to \NMCGoid(P) \int \mathscr{A}_0$.  This
next lemma now completes the proof of Theorems \ref{ribbon-graphs} and
\ref{mobius-graphs}.

\begin{lemma}\label{ribbon-final-lemma}
The functors $\phi'$ and $\widetilde{\phi}'$ are equivalences of categories.
\end{lemma}
\begin{proof}
  By cutting surfaces along complete arc systems we see that $\phi'$
  and $\widetilde{\phi}'$ are both essentially surjective and full.
  Faithfulness follows from the fact that the automorphism group of any
  admissible disc (i.e., an object of $\MCGoid_0(P)$ or
  $\NMCGoid_0(P)$) is trivial.
\end{proof}

\section{Proof of Theorem \ref{handlebody-theorem}}\label{handlebody-proof}

The argument begins along the same lines as the proof of Theorems
\ref{ribbon-graphs} and \ref{mobius-graphs} in the previous section.
We replace surfaces with handlebodies and arcs with discs.  The only
significant difference is in the analogue of Lemma
\ref{ribbon-final-lemma}, where the analogous functor is not an
equivalence of categories, but a somewhat more involved argument shows
that it is nevertheless still a homotopy equivalence on classifying
spaces.

\subsection{Outline of the proof}
The universal property for the homotopy left Kan extension applied to
the inclusion of cyclic operads $\Hand_0 \hookrightarrow \Hand$
provides a morphism of modular operads
\[
\Theta\co \mathbb{L}\Mod_!\Hand_0 \to \Hand
\]
that we will show to be a homotopy equivalence (except on the unmarked
solid torus component).  Since both sides are determined by their
values on corollas, it suffices to check that $\Theta$ is a homotopy
equivalence whenever one evaluates both modular operads on a corolla
$*_P$.

We first describe $\Theta$ at the level of categories and groupoids.
Recall that $\Hand(*_P) = B\mathscr{H}(P)$ where $\mathscr{H}(P)$ is
the groupoid of admissible oriented handlebodies with parametrised
boundary discs labelled by $P$ (plus a solid torus if $P=\emptyset$),
and $\Hand_0(*_P) = B\mathscr{H}_0(P)$, where $\mathscr{H}_0(P)$ is
the subgroupoid of genus zero handlebodies (i.e. balls with labelled
discs on their boundary).  By Thomason's Theorem, there is a homotopy
equivalence
\[
\mathbb{L}\Mod_!\Hand_0(*_P) \simeq B\left( \Gconn(P) \int
  \mathscr{H}_0 \right);
\]
note that an object of the Grothendieck construction above is a
connected graph $\tau$ with legs labelled by $P$ and each non-leg
vertex labelled by a genus zero handlebody with boundary discs
corresponding to the half-edges incident at the vertex.  The map $\Theta$
is (up to homotopy) induced by the functor
\[
\theta \co \Gconn(P) \int \mathscr{H}_0 \to \mathscr{H}(P)
\]
that glues the genus zero handlebodies at the vertices together as
prescribed by $\tau$ to form an object of the groupoid
$\mathscr{H}(P)$.  

We will lift $\theta$ to a refined functor $\theta'$ mapping into a
homotopy equivalent category that is more amenable to comparison,
namely the category of handlebodies equipped with collections of
2-discs that divide them into genus zero admissible pieces.  The
source of $\theta'$ projects to the category $\Gconn(P)$ by forgetting
the labellings of the vertices.  The target category will also project
to $\Gconn(P)$ by taking the dual graph to a complete disc system.  We
will prove that $\theta'$ is a homotopy equivalence by composing these
projections with the graph reduction functor $R\co \Gconn(P) \to
\Gconnred(P)$ and making a fibrewise comparison over $\Gconnred(P)$.

\subsection{Disc systems}\label{disc-systems-section}

Let $K=K_{g,n}$ be a connected handlebody of genus $g$ with $n$ discs
marked on its boundary.  A \textbf{cutting disc} in $K$ is a disc $D$
properly embedded in $K$ (i.e. the interior is in the interior of $K$
and the boundary is in boundary of $K$) with $\partial D$ disjoint
from the $n$ marked discs in $\partial K$.  We may cut $K$ along $D$
to obtain a (possibly disconnected handlebody) with $n+2$ marked discs
on its boundary coming from the original $n$ discs plus an additional
disc on either side of the cut.

\begin{definition}
  A \textbf{disc system} in $K$ is a set of disjoint cutting discs
  $D_i$ in $K$ such that if one cuts $K$ along all of the $D_i$ then
  each component of the result is an admissible handlebody.  A disc
  system in $K$ is said to be \textbf{complete} if the result of
  cutting $K$ along the discs in the system consists of only genus
  zero pieces.
\end{definition}

Given a disc system $\alpha$ in $K$, there are two associated groups.
The first is the group 
\[
\mathrm{Stab}(\alpha) \subset \mathcal{H}(K)
\]
of isotopy classes of diffeomorphisms of $K$ fixing the marked
boundary discs and disc system $\alpha$ (more precisely, the
diffeomorphism must send $\alpha$ to itself without permuting the
discs of the system).  To define the second group, let $K_\alpha$
denote the result of cutting $K$ along the discs of $\alpha$, and
write $\{K_{\alpha,j}\}$ for the components of $K_\alpha$; we regard
each component $K_{\alpha,j}$ as a handlebody with marked boundary
discs.  The second group is
\[
\prod_j \mathcal{H}(K_{\alpha,j}).
\]

\begin{proposition}\label{gluing-kernel-prop}
  Gluing the components of $K_\alpha$ back together induces a
  homomorphism
\[
\prod_j \mathcal{H}(K_{\alpha,j}) \to \mathrm{Stab}(\alpha).
\]
This homomorphism is surjective, and its kernel is the free abelian
subgroup freely generated by the elements that simultaneously Dehn
twist in opposite directions around each of the cuts.
\end{proposition}

As with arc systems, isotopy classes of disc systems in $K$ form a
category $\mathscr{D}(K)$ in which there is an arrow $\alpha \to
\beta$ for each way of identifying $\beta$ with a subset of the discs
of $\alpha$ up to isotopy.  Let $\mathscr{D}_0(K)$ denote the
subcategory of complete disc systems.

\begin{proposition}\label{contractibility-lemma}
  If $K$ is admissible then $B\mathscr{D}(K) \simeq B\mathscr{D}_0(K)
  \simeq *$.
\end{proposition}
\begin{proof}
  Contractibility of $B\mathscr{D}(K)$ is the content of \cite[Theorem
  5.3]{McCullough} --- the proof given there does not mention the
  setting in which $K$ has a nonzero number of marked boundary discs,
  but the same argument works without modification.  Alternatively,
  one can adapt the argument of \cite{Hatcher-triang} to discs.
  Contractibility of $B\mathscr{D}_0(K)$ then follows from an
  induction completely analogous to the proof of Lemma
  \ref{arc-contractibility}.
\end{proof}
(In the exceptional case of a solid torus, neither $B\mathscr{D}(K)$ nor
$B\mathscr{D}_0(K)$ is contractible).

We can regard $K \mapsto \mathscr{D}_0(K)$ as a functor
$\mathscr{D}_0\co \mathscr{H}(P) \to \Cat.$
One can then form the Grothendieck construction $\mathscr{H}(P) \int
\mathscr{D}_0$ Proposition \ref{contractibility-lemma} and Thomason's
Theorem then give,

\begin{lemma}
  The projection $\mathscr{H}(P) \int \mathscr{D}_0 \to
  \mathscr{H}(P)$ is a homotopy equivalence (except on the component
  corresponding to a solid torus when $P=\emptyset$).
\end{lemma}

The functor $\theta \co \Gconn(P)\int \mathscr{H}_0 \to \mathscr{H}(P)$
lifts to
\[
\theta' \co \Gconn(P)\int \mathscr{H}_0 \to \mathscr{H}(P) \int \mathscr{D}_0
\]
as follows.  An object in the source category is a graph with vertices
labelled by genus zero handlebodies.  An object of the target is a
handlebody equipped with a complete disc system.  The functor $\theta$
was defined by gluing the genus zero handlebodies at the vertices
together as prescribed by the graph; we define $\theta'$ by the same
procedure, except that we now remember where the gluings were performed
to obtain a complete disc system in the resulting handlebody.

\subsection{From disc systems to dual graphs}

A complete disc system $\alpha$ in a handlebody $K$ with boundary
discs labelled by $P$ determines a \textbf{dual graph} $\Gamma(\alpha)
\in \Gconn(P)$ as follows.  Let $K_\alpha$ denote the result of
cutting $K$ along the discs of $\alpha$.  To build the graph
$\Gamma(\alpha)$, put a vertex for each connected component of
$K_\alpha$ and an edge for each cutting disc $D$ of $\alpha$; the
endpoints of the edge corresponding to $D$ are the vertices
corresponding to the components on either side of $D$.  Finally, for
each $p\in P$ add a leg (an edge terminating in a univalent vertex)
with label $p$ at the vertex corresponding to the component on which
the corresponding boundary disc lies.  Sending a complete disc system
to its dual graph defines a functor
\[
\Gamma\co \left(\mathscr{H}(P) \int \mathscr{D}_0 \right) \to \Gconn(P).
\]

Consider $\mathbb{L}\pi_! pt \co \Gconn(P) \to \Top$, where $\pi\co
\Gconn(P) \int \mathscr{H}_0 \to \Gconn(P)$ is the projection functor,
and $pt$ is the singleton-valued constant functor.  The functor
$\theta'$ induces a natural transformation
\[
\theta'' \co \mathbb{L}\pi_! pt \to \mathbb{L}\Gamma_! pt.
\]
Note that this transformation is not generally a pointwise homotopy equivalence
of functors. Let us examine these functors more closely.

\begin{lemma}\label{identification-of-groups-lemma}
Let $\tau$ be a graph in $\Gconn(P)$.  Choose a handlebody $K \in
\mathscr{H}(P)$ and a disc system $\alpha$ in $K$ such that the dual
graph $\Gamma(\alpha)$ is isomorphic to $\tau$. 
\begin{enumerate}
\item $\mathbb{L}\pi_! pt (\tau) \simeq B(\pi \commacat \tau) \simeq
  \prod_j B\mathcal{H}(K_{\alpha,j})$, and edge contractions
  correspond to gluings.
\item $\mathbb{L}\Gamma_! pt (\tau) \simeq B(\Gamma \commacat \tau)
  \simeq B\mathrm{Stab}(\alpha)$.  Contracting an edge in $\tau$
  corresponds to the morphism $\alpha \to \alpha'$ that forgets the
  corresponding cutting disc, and the induced map is given by
  including the stabiliser of $\alpha$ into the stabiliser of
  $\alpha'$.
\item Under these homotopy equivalences, the map $\theta'' \co
  \mathbb{L}\pi_! pt (\tau) \to \mathbb{L}\Gamma_! pt (\tau)$ corresponds to
  the map of Proposition \ref{gluing-kernel-prop}.
\end{enumerate}
\end{lemma}
\begin{proof}
  For (1) and (2), the first equivalence follows from the definitions.
  The equivalence $B\mathrm{Stab}(\alpha) \stackrel{\simeq}{\to}
  B(\Gamma \commacat \tau)$ is induced by the inclusion $J$ of the
  automorphism group of the object $\overline{\alpha} =
  (\alpha,\Gamma(\alpha) \cong \tau) \in (\Gamma \downarrow \tau)$.
  This inclusion is a homotopy equivalence of classifying spaces by
  Quillen's Theorem A since the comma category $(\overline{\alpha}
  \commacat J)$ has a final object.  The second equivalence in (1) is
  shown by a variation of the same reasoning.  Statement (3) is a
  matter of unwinding the definitions.
\end{proof}

\subsection{Integrating out the bivalent vertices: from graphs to reduced graphs}
\label{bivalent-verts-section}

To prove that $\theta' \co \Gconn(P)\int \mathscr{H}_0 \to
\mathscr{H}(P) \int \mathscr{D}_0$ is a homotopy equivalence we shall
examine it fibrewise over $\Gconnred(P)$ via the graph reduction
functor $R\co \Gconn(P) \to \Gconnred(P)$ that removes the inessential
bivalent vertices.

Observe that for any operad $\mathcal{O}$ in a symmetric monoidal
category $(\mathscr{C},\otimes)$, the object $\mathcal{O}(*_2)$
associated with a bivalent corolla $*_2$ is a monoid in $\mathscr{C}$;
the product is induced by the edge contraction shown below.
\begin{center}
\includegraphics{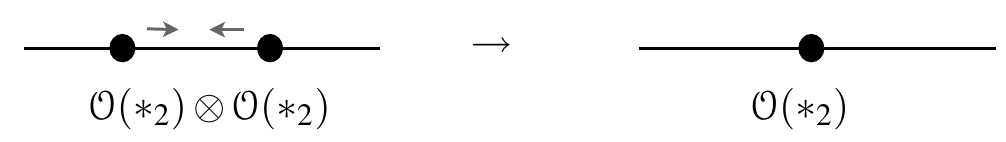}
\end{center}
Moreover, for each $p \in P$, there is an action of this monoid on the
object $\mathcal{O}(*_P)$ induced by gluing a bivalent vertex at the
$p$ leg and then contracting it in to the centre.

In the case of $B\mathscr{H}_0$, the topological monoid associated
with a bivalent corolla is homotopy equivalent to the circle group
$S^1$ since the groupoid $\mathscr{H}_0(*_2)$ is equivalent to $\Z$
(generated by the Dehn twist around the equator of a 3-ball with
marked boundary discs at the north and south poles).  Thus the space
$B\mathscr{H}(P)$ is equipped with $|P|$ circle actions.  Given an
internal edge $e$ in a graph $\tau$, let $S^1_e$ denote the circle
acting on $\mathbb{L}\pi_! pt (\tau)$ via the diagonal of the two circle
actions coming from the two endpoints of $e$, as depicted below.
\begin{center}
\includegraphics{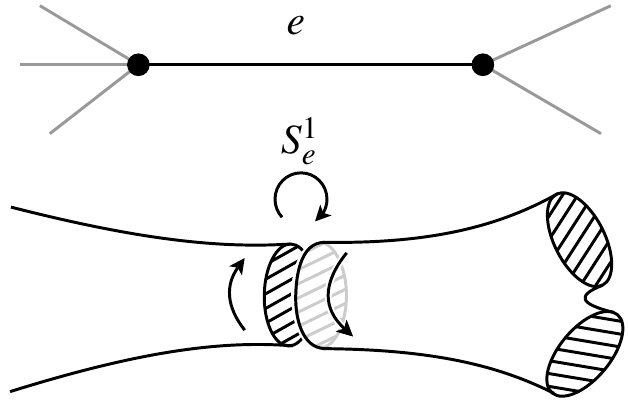}
\end{center}

\begin{lemma}
$\mathbb{L}R_! \mathbb{L}\pi_! pt (\tau) \simeq \mathbb{L}\pi_! pt (\tau)
\times_{\prod_e S^1_e} E\left(\prod_e S^1_e\right).$
\end{lemma}
\begin{proof}
  Let $k=|E_{\mathit{int}}(\tau)|$.  By Propositions
  \ref{bivalent-to-simplicial-prop} and \ref{simplicial-equiv-prop},
  $\mathbb{L}R_! \mathbb{L}\pi_! pt (\tau)$ is homotopy equivalent to
  the realisation of a $k$-fold semi-simplicial space.  Furthermore,
  by inspection one sees that this $k$-fold semi-simplicial space is
  exactly the iterated 2-sided bar construction for each of the
  circles $S^1_e$ acting as above.
\end{proof}

\begin{lemma}
  For $\tau \in \Gred$ there is a homotopy equivalence $\mathbb{L}R_!
  \mathbb{L}\Gamma_! pt (\tau) \simeq \mathbb{L}\Gamma_!  pt (\tau)$, and
  this is natural in $\tau$.
\end{lemma}
\begin{proof}
  Suppose $K$ is a handlebody with marked boundary discs, $\alpha$ is
  a complete disc system in $K$, and $D \subset K$ is a disc of
  $\alpha$ that is parallel to either one of the marked boundary discs
  or another disc of $\alpha$.  Then $D$ corresponds to an edge with
  an inessential endpoint in the dual graph $\Gamma(\alpha)$, and
  forgetting $D$ prodces a complete disc system $\overline{\alpha}$
  whose dual graph is obtained from $\Gamma(\alpha)$ by contracting
  the edge corresponding to $D$.  An element of $\mathcal{H}(K)$ acts
  trivially on the dual graph of $\overline{\alpha}$ if and only if it
  acts trivially on the dual graph of $\alpha$.  It follows from this
  and Lemma \ref{identification-of-groups-lemma} that
  $\mathbb{L}\Gamma_! pt$ sends each arrow of $(R\commacat \tau)_0$ to
  a homotopy equivalence.  The result now follows from Propositions
  \ref{bivalent-to-simplicial-prop} and \ref{simplicial-equiv-prop}.
\end{proof}

Combining the above two lemmas with Proposition \ref{gluing-kernel-prop} and
Lemma \ref{identification-of-groups-lemma} now yields the key lemma of
the proof.

\begin{lemma}\label{final-lemma}
  The natural transformation $\theta''$ induces a homotopy equivalence
  of functors, $\mathbb{L}R_! \mathbb{L}\pi_! pt \stackrel{\simeq}{\to}
  \mathbb{L}R_! \mathbb{L}\Gamma_! pt.$
\end{lemma}

We can now complete the proof.  Up to homotopy, the map $\Theta \co
\mathbb{L}\mathrm{Mod}_!  \Hand_0(*_P) \to \Hand(*_P)$ factors as a
sequence of equivalences
\begin{align*}
\mathbb{L}\mathrm{Mod}_!  \Hand_0(*_P) & \simeq \hocolim_{\Gconn(P) \int \mathscr{H}_0} pt \\
& \simeq \hocolim_{\Gconnred(P)} \mathbb{L}R_! \mathbb{L}\pi_! pt \\
& \stackrel{\simeq}{\to} \hocolim_{\Gconnred(P)}
\mathbb{L}R_!\mathbb{L}\Gamma_! pt \\
& \simeq B\left(\mathscr{H}(P) \int \mathscr{D}_0\right) \stackrel{\simeq}{\to}
B\mathscr{H}(P) \simeq \Hand(*_P),
\end{align*}
(of course, when $P$ is empty, the composition fails to be a homotopy
equivalence on the component corresponding to a solid torus).  This
completes the proof.

\section{Cohomology of modular envelopes}\label{cohomology-section}

\subsection{The Bousfield-Kan spectral sequence for a modular envelope}
Expressing a modular operad $\mathcal{M}$ as the derived modular
envelope of a cyclic operad $\mathcal{O}$ gives, in particular, a
description of the individual spaces of $\mathcal{M}$ in terms of
homotopy colimits of products of spaces of $\mathcal{O}$.  This leads
to a spectral sequence computing the homology of $\mathcal{M}$.

In general, let $F\co \mathscr{C} \to \Top$ be a functor.  There is a
Bousfield-Kan homology spectral sequence,
\[
E^2_{p,q} = {\colim}^p H_q(F; k) \Rightarrow H_{p+q}(\hocolim F; k),
\]
where $\colim^p$ is the $p$th left derived functor of
$\colim_\mathscr{C}$.  The functor $F$ is said to be \emph{formal}
over a field $k$ if there is a zigzag of natural transformations
between the functors $C_*(F;k)$ and $H_*(F; k)$ from $\mathscr{C}$ to
the category $\dgVect$ of chain complexes over $k$ (where the values
of the second functor are given the zero differential).

\begin{proposition}
If $F\co \mathscr{C} \to \Top$ is a functor that is formal over $k$ then
the Bousfield-Kan spectral sequence for $H_*(\hocolim F; k)$ collapses
at the $E^2$ page.
\end{proposition}
\begin{proof}
  The space $\hocolim F$ is the geometric realisation of the
  bisimplicial set
\[
(p,q) \mapsto \bigoplus_{\sigma \in N_p\mathscr{C}} S_q F(\sigma(0)),
\]
where $N_p\mathscr{C}$ is the set of $p$-simplices in the nerve of
$\mathscr{C}$, $S_q(-)$ is the set of singular $q$-simplices and
$\sigma(0)$ is the first vertex of $\sigma$.  Hence the homology of
$\hocolim F$ is computed by the total complex of the
associated double complex
\[
B_{p,q} = \bigoplus_{\sigma \in N_p\mathscr{C}} C_q( F(\sigma(0)); k)
\]
The Bousfield-Kan homology spectral sequence comes from the horizontal
filtration on this double complex.  If $F$ is formal then the double
complex $B$ is quasi-isomorphic to the double complex $B'$
obtained by replacing $C_*(F; k)$ with $H_*(F;k)$. The spectral
sequences for the horizontal filtrations on $B$ and $B'$ are
isomorphic from the $E^2$ page onwards.  Since the vertical
differential in $B'$ is identically zero, the $B'$ spectral sequence
collapses at $E^1$.
\end{proof}

Let $K_{g,n}$ be a 3-dimensional compact connected oriented handlebody of genus $g$ with
$n$ marked boundary discs.  By Theorem \ref{handlebody-theorem},
\[
B\Diff(K_{g,n}) \simeq \hocolim_{\G_{g,n}} \Hand_0
\]
where $\G_{g,n}$ is the full subcategory of $\Gconn(n)$ consisting of
graphs of rank $g$.  This gives a Bousfield-Kan homology spectral
sequence
\begin{equation}\label{BKSS}
E^2_{p,q} = {\colim_{\G_{g,n}}}^p H_q(\Hand_0; k) \Rightarrow
H_{p+q}(B\Diff(K_{g,n}); k)
\end{equation}
The main result of \cite{cyclic-formality} is that the cyclic operad
$\Hand_0$ is formal over $\mathbb{R}$.  Hence,
\begin{corollary}
  When $k=\mathbb{R}$ the above spectral sequence \eqref{BKSS}
  collapses at the $E^2$ page.
\end{corollary}

\subsection{Graph homology and modular envelopes}

Let us recall the definition of the graph homology complexes
$C\Gamma_*^\mathcal{O}(g,n)$ for a cyclic operad $\mathcal{O}$ in the
category $\dgVect$ of chain complexes over a field $k$.  See
\cite{Lazarev-Voronov} and \cite{Conant-Vogtmann} for further details.

If $V$ is a finite dimensional vector space then $\mathrm{Det}(V)$
denotes the 1-dimensional vector space given by the top exterior power
$\Lambda^{\mathrm{dim}(V)}V$.  The orientation line of a graph
$\gamma$ is defined as,
\[
\mathrm{Or}(\gamma) \defeq  \mathrm{Det}(k^{E_{int}(\gamma)})\otimes
\mathrm{Det}(H^1(\gamma)),
\]
regarded as a graded vector space concentrated in degree
$|E_{int}(\gamma)| -1$.  For a non-loop internal edge $e$, let
$\gamma/e$ denote the graph resulting from contracting $e$.  There is
an isomorphism $(e \wedge -) \co \mathrm{Or}(\gamma/e)
\stackrel{\cong}{\to} \mathrm{Or}(\gamma)$.

For each pair of nonnegative integers $(g,n)$ (excluding $(0,0)$,
$(1,0)$, and $(0,1)$) there is a graph complex for $\mathcal{O}$ made
from graphs of rank $g$ with $n$ ordered legs.  It is defined as
\[
C\Gamma_*^{\mathcal{O}}(g,n) \defeq \bigoplus_{\gamma}
\mathcal{O}(\gamma)\otimes \mathrm{Or}(\gamma),
\]
where the summation is over isomorphism classes of graphs in
$\G_{g,n} $. The differential on this chain complex is the sum of the
internal differential $d_\mathcal{O}$ from $\mathcal{O}$ acting on the
first factor and an edge contracting differential $d_E$ acting on the
second factor by the formula,
\[
d_E(\gamma \otimes (e\wedge \omega)) = \sum_{e} \gamma/e \otimes \omega,
\]
where the summation runs over all non-loop internal edges.  It is easy
to see that a quasi-isomorphism of cyclic operads $\mathcal{O} \to
\mathcal{O}'$ induces a quasi-isomorphism of the associated graph
complexes.

\begin{theorem}
  $H^*(\hoMod_!\mathcal{O}(*_n)_g) \cong H\Gamma^{D\mathcal{O}}_{3g -
    4 - *}(g,n)$.
\end{theorem}
Here $D\mathcal{O}$ is the dg-dual cyclic operad (see
\cite{Getzler-cyclic} or \cite{Lazarev-Voronov} for the definition),
and $\hoMod_!\mathcal{O}(*_n)_g$ denotes the genus $g$ component of
the space $\hoMod_!\mathcal{O}(*_n)$ in the canonical genus grading of
the modular envelope.
\begin{proof}
  Consider the canonical projection $\pi\co \hoMod_!\mathcal{O}(*_n)_g
  \to B\G_{g,n}$.  The hypercohomology of the presheaf $\mathscr{F}\co
  U \mapsto C^*(\pi^{-1}U)$ on $B\G_{g,n}$ is isomorphic to the
  cohomology of $\hoMod_!\mathcal{O}(*_n)_g$.  By the results of
  \cite{Lazarev-Voronov}, the hypercohomology of $\mathscr{F}$ in
  degree $*$ is isomorphic to the graph homology of $D\mathcal{O}$ in
  degree $3g- 4 -*$.
\end{proof}

Combining Theorem \ref{handlebody-theorem} and the formality of the
cyclic operad $\Hand_0$ from \cite{cyclic-formality} with the above
theorem then gives,

\begin{corollary}
  $H^*(B\Diff(K_{g,n});\mathbb{R}) \cong H\Gamma^{D\mathcal{BV}}_{3g -
    4 - *}(g,n)$, where $\mathcal{BV} = H_*(\Hand_0; \mathbb{R})$ is
  the Batalin-Vilkovisky cyclic operad.
\end{corollary}

\bibliographystyle{amsalpha} 
\bibliography{biblio-clean}

\end{document}